\documentclass[a4paper,12pt]{article} 
\usepackage[parfill]{parskip}
\usepackage[titletoc,toc,title]{appendix}
\usepackage[margin=2cm]{geometry}
\usepackage{ mathrsfs }
\usepackage{amsthm}
\usepackage{amsmath}
\usepackage{amssymb}
\usepackage{graphicx}
\usepackage{mathtools}
\usepackage{subcaption}
\usepackage{enumerate}
\usepackage{enumitem}
\usepackage{xcolor}
\usepackage{tikz}
\usetikzlibrary{calc,intersections,through,backgrounds}
\usepackage{array}
\usepackage[bookmarks=true]{hyperref}
\hypersetup{
	colorlinks,
	linkcolor={red!70!black},
	citecolor={red!90!black},
	urlcolor={green!80!black}
	linktoc=all, 
}
\makeatletter
\newtheorem*{rep@theorem}{\rep@title}
\newcommand{\newreptheorem}[2]{%
	\newenvironment{rep#1}[1]{%
		\def\rep@title{#2 \ref{##1}}%
		\begin{rep@theorem}}%
		{\end{rep@theorem}}}
\makeatother
\newtheoremstyle{stylename}
{15pt} 
{15pt} 
{\itshape} 
{} 
{\bfseries} 
{.} 
{.5em} 
{} 
\theoremstyle{stylename}
\newtheorem{theorem}{Theorem}[section]
\newreptheorem{theorem}{Theorem}
\newtheorem{lemma}[theorem]{Lemma}

\newtheoremstyle{exampstyle}
{15pt} 
{15pt} 
{} 
{} 
{\bfseries} 
{.} 
{.5em} 
{} 
\theoremstyle{exampstyle} 
   
\newtheorem{example}[theorem]{Example}

\usepackage{titlesec}

\titleformat{\chapter}[display]
  {\normalfont\sffamily\huge\bfseries\color{black}}
  {\chaptertitlename\ \thechapter}{20pt}{\Huge}

\titleformat{\section}
  {\Large\bfseries\color{black}}
  {\thesection}{1em}{}

\usepackage[backend=biber,style=alphabetic,sorting=nyt]{biblatex} 
\renewbibmacro{in:}{} 
\DeclareFieldFormat{pages}{#1}
\DeclareFieldFormat[article]{number}{ no. #1}
\bibliography{bib2}

\begin{document}


\title {On automorphism groups of toroidal circle planes}
\author{Brendan Creutz, Duy Ho, G\"{u}nter F. Steinke}
 \maketitle
	
	\begin{abstract}
		Schenkel proved that the automorphism group of a flat Minkowski plane is a Lie group of dimension at most 6 and described planes whose automorphism group has dimension at least 4 or one of whose kernels has dimension 3. We extend these results to the case of toroidal circle planes. 
\end{abstract}

\section{Introduction}
 
The  classical Minkowski plane is the geometry of plane sections of the standard nondegenerate ruled quadric in real 3-dimensional projective space $\mathbb{P}_3(\mathbb{R})$. It is an example of a flat Minkowski plane, which is an incidence structure defined on the torus satisfying two axioms, namely the Axiom of Joining and the Axiom of Touching. Torodal circle planes are a generalization of flat Minkowski planes in the sense that  they are incidence structures on the torus that are only required to satisfy the Axiom of Joining, cf. Subsection 2.1. There are many examples of flat Minkowski planes, cf. \cite[Section 4.3]{gunter2001}. For an example of a proper toroidal circle plane (that is not a flat Minkowski plane), see \cite{polster1998b}.

The automorphism group of a flat Minkowski plane is well studied. Schenkel \cite{schenkel1980} showed that this group is a Lie group of dimension at most 6. She also determined flat Minkowski planes whose automorphism groups have dimension at least 4 or one of whose kernels (cf. Subsection 2.4) has dimension 3.  Such planes are isomorphic to  the classical Minkowski plane, a family constructed by Steinke \cite{gunter1985}, or a family generalised by Schenkel \cite{schenkel1980} from the construction by Hartmann \cite{hartman1981}. 

In this paper we extend these  structural results from flat Minkowski planes to toroidal circle planes. Our first main result is  
 
\begin{theorem} \label{tcplie} The automorphism group of a toroidal circle plane is a Lie group  with respect to the compact-open topology  of dimension at most 6. 
\end{theorem}

 Schenkel \cite{schenkel1980} gives an outline of a proof of Theorem \ref{tcplie} in the case of flat Minkowski planes. Her method is based on convergence of automorphisms on so-called tripods and follows a proof of an analogous result for spherical circle planes by Strambach \cite{strambach1970d}. This approach neither makes use of parallel classes provided by toroidal circle planes, which causes problems in covering the entire point set, nor directly yields the right upper bound for the dimensions of the automorphism groups. An alternative proof in \cite{gunter1984} implicitly uses the Axiom of Touching and therefore cannot be carried over to toroidal circle planes. We provide a simplified proof of Theorem \ref{tcplie} that avoids these problems.

Based on the classification by Schenkel and Theorem \ref{tcplie}, we also show that a toroidal circle plane will automatically satisfy the Axiom of Touching when its automorphism group is large, in the following sense.

\begin{theorem} \label{dimension4} Let $\mathbb{T}$ be a toroidal circle plane with full automorphism group 
	$\text{\normalfont Aut}(\mathbb{T})$. If $\dim \text{\normalfont Aut}(\mathbb{T}) \ge 4$ or one of its kernels is 3-dimensional, then $\mathbb{T}$ is a flat Minkowski plane.   
\end{theorem}

The problem of classifying toroidal circle planes with 3-dimensional automorphism groups remains open. In   \cite[Section 5.3]{duythesis}, it was shown that only certain groups can occur and their possible actions are determined. Many families of flat Minkowski planes with 3-dimensional automorphism groups are known, cf. \cite[Section 4.3]{gunter2001}, \cite{gunter2004}, \cite{gunter2017}, \cite[Chapter 6]{duythesis}. There are currently no known examples of proper toroidal circle planes  with 3-dimensional automorphism groups. The only known example of a proper toroidal circle plane \cite{polster1998b} has a 2-dimensional automorphism group, cf. \cite[Section 7.4]{duythesis}.

In Section 2, we recall some facts about  toroidal circle planes. The proofs of the main theorems are presented in Section 3. 
 
\section{Preliminaries} 
\subsection{Toroidal circle planes, flat Minkowski planes and examples}
	A \textit{toroidal circle plane} is a geometry $\mathbb{T}=(\mathcal{P}, \mathcal{C}, \mathcal{G}^+, \mathcal{G}^-)$, whose 
	\begin{enumerate}[label=$ $]
		\item point set $\mathcal{P}$ is  the torus $\mathbb{S}^1 \times \mathbb{S}^1$, 
		\item circles (elements of $\mathcal{C}$) are graphs of homeomorphisms of $\mathbb{S}^1 $,  
		\item $(+)$-parallel classes (elements of $\mathcal{G}^+$) are the verticals $\{ x_0 \}  \times \mathbb{S}^1$,  
		\item $(-)$-parallel classes (elements of $\mathcal{G}^-$) are the horizontals $\mathbb{S}^1 \times \{ y_0 \}$,  
	\end{enumerate} 
	where $x_0, y_0 \in \mathbb{S}^1$. 
	
	 We denote the $(\pm)$-parallel class containing a point $p$ by $[p]_\pm$. When two points $p,q$ are on the same $(\pm)$-parallel class, we say they are \textit{$(\pm)$-parallel}  and denote  this by $p \parallel_{\pm} q$. Two points $p,q$ are $\textit{parallel}$ if they are  $(+)$-parallel or $(-)$-parallel, and we denote this by $p \parallel q$. 
	 
	Furthermore, a toroidal circle plane satisfies the following
	\begin{enumerate}[label=]
		\item \textit{Axiom of Joining}: three pairwise non-parallel points   can be joined by a unique circle. 
	\end{enumerate}
  
	A toroidal circle plane is called a \textit{flat Minkowski plane}  if it also satisfies the following
	\begin{enumerate}[label=$ $]
		\item \textit{Axiom of Touching}: for each circle $C$ and any two nonparallel points $p,q$ with $p \in C$ and $q \not \in C$, there is exactly one circle $D$ that contains both points $p,q$ and intersects $C$ only at the point $p$.
	\end{enumerate}

There are various known examples of flat Minkowski planes. For our purpose, we describe two particular families, which play a prominent role in the classification of flat Minkowski planes. We identify $\mathbb{S}^1$ as $\mathbb{R} \cup \{ \infty \}$ in the usual way.

 \begin{example}[{Swapping half plane $\mathcal{M}(f,g)$, cf. \cite[Subsection 4.3.1]{gunter2001}}] \label{ex:swappinghalf} Let $f$ and $g$ be two orientation-preserving homeomorphisms of $\mathbb{S}^1$. Denote $\text{\normalfont PGL}(2,\mathbb{R})$ by $\Xi$ and $\text{\normalfont PSL}(2,\mathbb{R})$ by $\Lambda$. The circle set $\mathcal{C}(f,g)$ of a \textit{swapping half plane $\mathcal{M}(f,g)$} consists of sets of the form 
 	$$
 	\{ (x,\gamma(x)) \mid x \in \mathbb{S}^1 \},
 	$$
 	where $\gamma \in \Lambda \cup g^{-1}(\Xi \backslash \Lambda) f$. 
 \end{example}
 
 \begin{example}[{Generalised Hartmann plane $\mathcal{M}_{GH}(r_1,s_1;r_2,s_2)$, cf. \cite[Subsection 4.3.4]{gunter2001}}] \label{ex:hartmann}
 	 For $r,s>0$, let $f_{r,s}$ be the orientation-preserving \textit{semi-multiplicative homeomorphism} of $\mathbb{S}^1$ defined by
 	 $$f_{r,s}(x) = 
 	 \begin{cases}
 	 x^r   & \text{for } x\ge 0,\\
 	 -s|x|^r & \text{for } x<0, \\
 	 \infty  & \text{for } x=\infty. \\
 	 \end{cases} 
 	 $$
 	
The circle set  $\mathcal{C}_{GH}(f,g)$  of a \textit{generalised Hartmann plane $\mathcal{M}_{GH}(r_1,s_1;r_2,s_2)$} consists of sets of the form 
 	$$
 	\{(x,sx+t)\mid x \in \mathbb{R} \} \cup \{ (\infty,\infty)\},
 	$$
 	where $s,t \in \mathbb{R}$, $s \ne 0$,   sets of the form 
 	$$
 	\left\{ \left(x,\dfrac{a}{f_{r_1,s_1}(x-b)}+c \right) \;\middle|\;  x \in \mathbb{R} \right\}  \cup \{ (b,\infty),(\infty,c)\},
 	$$
 	where $a,b,c \in \mathbb{R}$, $a > 0$, and  sets of the form
 		$$
 		\left\{ \left(x,\dfrac{a}{f_{r_2,s_2}(x-b)}+c \right) \;\middle|\;  x \in \mathbb{R}  \right\}  \cup \{ (b,\infty),(\infty,c)\},
 		$$
 	where $a,b,c \in \mathbb{R}$, $a < 0$. 

 \end{example}

One can obtain the classical Minkowski plane as the swapping half plane $\mathcal{M}(id,id)$, where $id$ denotes the identity map; or alternatively, the generalised Hartmann plane $\mathcal{M}_{GH}(1,1;1,1)$.

 \subsection{Geometric operations}

For a metric space $(X,d)$, let $\mathcal{H}(X)$ be the set of all of its closed subsets.  The set $\mathcal{H}(X)$ is a metric space if we equip it with the Hausdorff metric induced from the metric $d$, which is   defined by
$$
\mathbf{h}(C,D) = \max \{ \sup_{y\in D} \inf_{x \in C} d(x,y),  \sup_{x\in C} \inf_{y \in D} d(x,y) \},
$$
for two closed subsets $C$ and $D$ of $X$, cf. \cite[430, 431]{gunter2001}.

Considering the point set $\mathcal{P} \cong \mathbb{S}^1 \times \mathbb{S}^1$ as a subset of $\mathbb{R}^3$, we equip $\mathcal{P}$ with the  metric $\mathbf{e}$  induced by the Euclidean metric of $\mathbb{R}^3$. The circle set $\mathcal{C}$ is equipped with the Hausdorff metric $\mathbf{h}$ induced from $\mathbf{e}$.

Let $\widetilde{\mathcal{P}^3}$ be the subspace of the product space $\mathcal{P}^3$ consisting of all triples of pairwise nonparallel points. 
Let $\mathcal{P}^{1,2}$ be the subspace of $\mathcal{H}(\mathcal{P}^2)$ defined by   $\mathcal{P}^{1,2} \coloneqq \{ \{x,y\} \mid x,y \in \mathcal{P} \}$.
Let   $\mathcal{C}^{2*}$ be the subspace of the product space $\mathcal{C}^2$ which consists of all pairs of distinct circles that have non-empty intersection. 
 Let  $C^{1*} \subset C^{2*}$ be the subspace of pairs of touching circles.

We define five geometric operations on toroidal circle planes  as follows. 
 	\begin{enumerate}
 		
 		\item \textit{Joining}  $\alpha:  \widetilde{\mathcal{P}^3} \rightarrow \mathcal{C} $ is defined by $\alpha(x,y,z)$ being the unique circle going through three pairwise nonparallel points $x,y,z$. 
 		
 		\item \textit{Parallel Intersection} $\pi: \mathcal{P} \times \mathcal{P} \rightarrow \mathcal{P} $ is defined as $\pi: (x,y) \mapsto [x]_+ \cap [y]_- $.
 		
 		\item \textit{Parallel Projection} $\pi^{+(-)}: \mathcal{P} \times \mathcal{C} \rightarrow P$ is defined as $\pi^{+(-)}: (x,C) \mapsto [x]_{+(-)} \cap C$.
 		
 		\item \textit{Intersection} $\gamma: \mathcal{C}^{2*} \rightarrow \mathcal{P}^{1,2}$ is defined as $\gamma(C,D)=C \cap D$.
 		
 		\item \textit{Touching} $\beta: C^{1*} \rightarrow \mathcal{P}$ is defined as $\beta(C,D)=C\cap D$.

 	\end{enumerate} 
 
 The  geometric operations are continuous if and only  if $\mathcal{C}$ is equipped with  the topology  induced by the metric  $\mathbf{h}$.  For a proof, see \cite[Section 3.4]{duythesis}.
 
 Toroidal circle planes satisfy the K4 coherence condition for flat Minkowski planes introduced
 by Schenkel \cite[\text{2.1}]{schenkel1980}. 
\begin{lemma}[K4 coherence condition, special case] \label{lK4}
	
	Let $(C_n) \in \mathcal{C}$ be a sequence of circles and $(p_{i,n}) \rightarrow p_i, i =1,2,3$ be three converging sequences of points  such that $p_{i,n} \in C_n$, $p_1 \parallel_+ p_2$ and $p_3 \not \in [p_1]_+$. 
	
	Suppose  $(p_n) \rightarrow p \subseteq \mathcal{P}$ such that $p$ is not parallel to any $p_{i}$.  Then $(\pi^- (p_n, C_n)) \rightarrow \pi(p_1,p)$ and $(\pi^+ (p_n, C_n)) \rightarrow \pi(p,p_3)$.
\end{lemma}

For a proof, see \cite[\text{3.8}]{schenkel1980} or \cite[Lemma 4.1.3]{duythesis}.

\subsection{Derived planes} 
	The \textit{derived plane $\mathbb{T}_p$ of $\mathbb{T}$ at the point $p$} is the incidence geometry whose point set is $\mathcal{P} \backslash ( [p]_+ \cup [p]_-)$, whose lines are all parallel classes not going through $p$ and all circles of $\mathbb{T}$ going through $p$.   
For every point $p \in \mathcal{P}$, the derived plane $\mathbb{T}_p$ is an $\mathbb{R}^2$-plane  and even a flat  affine plane   when $\mathbb{T}$ is a flat Minkowski plane, cf. \cite[Theorem 4.2.1]{gunter2001}.

$\mathbb{R}^2$-planes were introduced by Salzmann in the 1950s and have been thoroughly studied, see \cite[Chapter 31]{salzmann1995}  and references therein. Following the notations from \cite[\text{31.4}]{salzmann1995}, we denote the line going through two points $a$ and $b$ by $ab$. The \textit{closed interval} $[a,b]$ is the intersection of all connected subsets of the line $ab$ that contain $a$ and $b; \textit{open intervals}$ are defined by $(a,b)=[a,b] \backslash \{a,b\}$.

One particular result on $\mathbb{R}^2$-planes we will use is the following.
\begin{lemma}[{cf. \cite[Proposition 31.12]{salzmann1995}}] \label{lcollinear}
	In $\mathbb{R}^2$-planes, collinearity and the order of (collinear) point triples are preserved under limits, in the following sense:
	
	\begin{enumerate}[label=(\alph*)]
		\item If the point sequences $(a_n),(b_n),(c_n)$ have mutually distinct limits $a,b,c$ and if $a_n,b_n,c_n$ are collinear for infinitely many $n \in \mathbb{N}$, then $a,b,c$ are collinear as well.
		\item If, in addition, $b_n \in (a_n,c_n)$ for infinitely many $i \in \mathbb{N}$, then $b \in (a,c)$. 
	\end{enumerate}
\end{lemma}

The geometric operations on the derived  $\mathbb{R}^2$-plane $\mathbb{T}_p$ at a point $p$ are induced from those on $\mathbb{T}$. On $\mathbb{T}_p$, we coordinatize the point set $\mathbb{R}^2$ in the usual way. It will be convenient to use the maximum metric  $\mathbf{d}$  defined as
$$
\mathbf{d}(p,q):=\max(|x_1-x_2|,|y_1-y_2|),
$$
for given two points $p:=(x_1,x_2), q:=(y_1,y_2)$ in $\mathbb{R}^2$. This metric is equivalent to restriction of $\mathbf{e}$ to $\mathbb{T}_p$.

\subsection{The automorphism group}

An \textit{isomorphism between two toroidal circle planes} is a bijection between the point sets that maps circles to circles, and induces a bijection between the circle sets. It can be shown that parallel classes are mapped to parallel classes.  

An \textit{automorphism of a   toroidal circle plane $\mathbb{T}$} is an isomorphism from $\mathbb{T}$ to itself.   With respect to composition, the set of all automorphisms of a toroidal circle plane is an abstract group. We denote this group by $\text{Aut}(\mathbb{T})$. Every automorphism of a toroidal circle plane is continuous and thus a homeomorphism of the torus, cf. \cite[Theorem 4.4.1]{gunter2001}.


 Let $C(\mathcal{P})$ be the space of continuous mappings from $\mathcal{P}$ to itself.
 As each automorphism is an element of $C(\mathcal{P})$, it is natural to equip the automorphism group $\text{Aut}(\mathbb{T})$ with the compact-open topology. Let
 $$
 (A,B)=\{ \sigma \in C(\mathcal{P}) \mid \sigma(A) \subset B\},
 $$
 where $A \subset \mathcal{P}$ is compact and $B \subset \mathcal{P}$ is open. The collection of all sets of the form $(A,B)$ is a subbasis for the compact-open topology of $C(\mathcal{P})$.
 
 Since $\mathcal{P}$ is a compact metric space, the compact-open topology is  equivalent to the topology of uniform convergence, cf.  \cite[283, 286]{munkres1974}.  
 Furthermore,  the compact-open topology is metrisable and a metric is given by
 $$
 \widetilde{\mathbf{e}}(\sigma,\tau)= \sup \{ \mathbf{e}( \sigma(x), \tau(x)) \mid  x \in \mathcal{P} \}.
 $$
 
With respect to the  topology induced from $C(\mathcal{P})$, the group $\text{\normalfont Aut}(\mathbb{T})$ becomes a topological group (cf. \cite[Theorem 4]{arens1946} or \cite[\text{96.6, 96.7}]{salzmann1995}), with a countable basis (cf. \cite[Theorem XII.5.2]{dugundji1966}).


The automorphism group $\text{\normalfont Aut}(\mathbb{T})$ has two distinguished normal subgroups, the \textit{kernels} $T^+$ and $T^-$ of the action of $\text{\normalfont Aut}(\mathbb{T})$ on the set of parallel classes $\mathcal{G}^+$ and $\mathcal{G}^-$, respectively. In other words, the kernel $T^\pm$ consists of all automorphisms of $\mathbb{T}$ that fix every $(\pm)$-parallel class. For convenience, we  refer to these two subgroups as the \textit{kernels $T^\pm$  of the plane $\mathbb{T}$}. 

%
%


For flat Minkowski planes, Schenkel  obtained the following results.
\begin{theorem} \label{minklie} Let $\mathbb{M}$ be a flat Minkowski plane. Then the automorphism group $\text{\normalfont Aut}(\mathbb{M})$ is a Lie group with respect to the compact-open topology of dimension at most 6. 
	
	\begin{enumerate}[label=(\alph*)]
	
\item	If $\text{\normalfont Aut}(\mathbb{M})$ has dimension at least 5, then  $\mathbb{M}$ is isomorphic to the classical flat Minkowski plane.
	
\item	If $\text{\normalfont Aut}(\mathbb{M})$ has dimension 4, then $\mathbb{M}$  is isomorphic to one of the following planes. 
	
	\begin{enumerate}[label=(\roman*)]
		\item A nonclassical swapping half plane $\mathcal{M}(f,id)$, where $f$ is a semi-multiplicative homeomorphism of the form $f_{d,s}$, $(d,s) \ne (1,1)$. This plane admits the $4$-dimensional group of automorphisms
		$$
		\{ (x,y) \mapsto (rx,\delta(y)) \mid r \in \mathbb{R}^+,\delta \in \text{\normalfont PSL}(2,\mathbb{R})\}. 
		$$
		
		\item A nonclassical generalised Hartmann plane $\mathcal{M}_{GH}(r_1,s_1;r_2,s_2)$, where $r_1,s_1,r_2,s_2 \in \mathbb{R}^+,$ $(r_1,s_1,r_2,s_2) \ne (1,1,1,1)$. This plane admits the 4-dimensional group of automorphisms 
		$$
		\{ (x,y) \mapsto (rx+a,sy+b) \mid a,b,r,s \in \mathbb{R}, r,s>0\}.  
		$$
	\end{enumerate}
		
\item	If one of the kernels of $\mathbb{M}$ is 3-dimensional, then $\mathbb{M}$ is isomorphic to a plane $\mathcal{M}(f,id)$, where $f$ is an orientation-preserving homeomorphism of $\mathbb{S}^1$. 
			
	\end{enumerate}

\end{theorem}

The proof of Theorem \ref{minklie} can be found in \cite[Chapter 5]{schenkel1980}, or \cite[Theorems 4.4.10, 4.4.12,  and 4.4.15]{gunter2001}.

 
%
%

\section{Proofs of Main Theorems}


Let $\Sigma$ be the subgroup of $\text{Aut}(\mathbb{T})$ consisting of all automorphisms that leave the sets $\mathcal{G}^\pm$ invariant and
preserve the orientation of parallel classes.   We first have the following. 
 
 \begin{lemma}\label{idmap}
 	In $\Sigma$, the identity map  is the only automorphism that fixes three pairwise nonparallel points. 
 \end{lemma}
 
 For a proof, cf. \cite[Lemma 4.4.5]{gunter2001}. 
 
 To prove Theorem \ref{tcplie}, we will verify that $\Sigma$ is locally compact by showing $\Sigma$ is homeomorphic to a closed subset of the locally compact space $\widetilde{\mathcal{P}^3}$ of dimension 6. Then, by a theorem of Szenthe, $\Sigma$ is a Lie group. Since $\Sigma$ has finite index in $\text{Aut}(\mathbb{T})$, it follows that $\text{Aut}(\mathbb{T})$ is also a Lie group of dimension at most 6.

Let $\widetilde{d}=(d_1,d_2,d_3) \in \widetilde{\mathcal{P}^3}$ be a triple of pairwise nonparallel points. Let $d_4:= \pi(d_2,d_3), d_5:=\pi(d_3, d_2)$, where $\pi$ denotes the operation Parallel Intersection (cf. Subsection 2.2).   In $\Sigma$, let $(\sigma_n)$ be a sequence of automorphisms that converges on the points $d_i$, and let $e_i := \lim_n \sigma_n(d_i)$, which we may assume to be also pairwise nonparallel,  for $i=1,\dots,5$.

 \begin{lemma} \label{denseset} In the derived plane $\mathbb{T}_{d_1}$, by means of joining, intersecting, parallel intersecting and parallel projecting, the two points $d_2, d_3$ generate a dense subset $\mathcal{D}$.
 \begin{proof}
 	It is sufficient to show that the set $\mathcal{D}$ generated by these operations is dense in $[d_2,d_5]$. Suppose for a contradiction that there exists an open interval  $(a,b) \subset [d_2,d_5] \backslash \overline{\mathcal{D}}$. Without loss of generality, we may assume $a \in (d_2,b)$, as in Figure \ref{fig:lemmadense}. We aim to construct a point in $(a,b) \cap \mathcal{D}$ by means of geometric operations.
	\begin{figure}[h]
		\begin{center}
			\begin{tikzpicture}[scale=0.8]
			
			\draw [fill] (0,0) circle [radius=0.1];
			\node[below left] at (0,0) {$d_2$};
			\coordinate (d_2) at (0,0);
			
			\draw [fill] (0,7) circle [radius=0.1];
			\node[above left] at (0,7) {$d_4$};
			\coordinate (d_4) at (0,7);

			\draw [fill] (6,0) circle [radius=0.1];
			\node[below right] at (6,0) {$d_5$};
			\coordinate (d_5) at (6,0);

			\draw [fill] (6,7) circle [radius=0.1];
			\node[above right] at (6,7) {$d_3$};
			\coordinate (d_3) at (6,7);
			
			\draw [fill] (2,0) circle [radius=0.1];
			\coordinate (a) at (2,0);
			\node[label={[label distance=0.05cm]270:$a$}] at (a) {};
			
			\draw [fill] (1.5,0) circle [radius=0.1];
			\coordinate (a_n) at (1.5,0);
			\node[label={[label distance=0.05cm]270:$a_n$}] at (a_n) {};
			
			\draw [fill] (5,0) circle [radius=0.1];
			\node[below] at (5,-0.1) {$b_n$};
			\coordinate (b_n) at (5,-0.1);
			
			\draw [fill] (4.5,0) circle [radius=0.1];
			\node[below] at (4.5,-0.1) {$b$};
			\coordinate (b) at (4.5,0);
			
			\draw  (d_4) -- (d_3);
			\draw  (d_5) -- (d_3);
			\draw  (d_4) -- (d_2);
			\draw  (d_2) -- (d_5);

			\draw  (a_n) -- (d_3);
			\draw (b_n) -- (d_4);

			
			\path [name path=and3] (a_n) -- (d_3);
			\path [name path=bnd4] (b_n) -- (d_4);
			\path [name intersections={of=and3 and bnd4,by=f_n}];
			\node[right] at (f_n) {$f_n$};
			\draw [fill] (f_n) circle [radius=0.1];
			
			\path [name path=d2d5] (d_2) -- (d_5);
			
			\coordinate (dum) at (0,3);
			
			\path [name path=fncn]  (f_n) -- +($(d_2)-(dum)$);
			
			\path [name intersections={of=d2d5 and fncn,by=c_n}];

			\node[label={[label distance=0.05cm]270:$c_n$}] at (c_n) {};
			
			\draw [fill] (c_n) circle [radius=0.1];

			\draw [dashed]  (f_n) -- (c_n);

			\end{tikzpicture}
		\end{center}
		\caption{} \label{fig:lemmadense}
	\end{figure} 	
 	
 	Let $(a_n)$ and $(b_n)$ be sequences convergent to $a$ and $b$, respectively, such that $a_n \in\mathcal{D} \cap [d_2,a)$, and  $b_n \in \mathcal{D} \cap (b,d_5]$. Let $f_n:= a_nd_3 \cap d_4b_n$ and let $c_n :=\pi(f_n,d_2)$. Then $c_n \in (a_n,b_n)$. Since the geometric operations are continuous, 
	the limits $f:=\lim_n f_n$ and  $c:= \lim_n c_n$ exist and $c \in [a,b]$. 
 	
 	We prove   $c \ne a$. Suppose the contrary, so that  $f \parallel_+ a$. Consider the sequence $(C_n)$ of circles $C_n:= \alpha(d_3,f_n,a_n)$ and the constant sequence $(d_1)$ on $\mathbb{T}$. By Lemma \ref{lK4},   $(\pi^-(d_1,C_n)) \rightarrow \pi(a,d_1) \ne d_1$, which contradicts the continuity of joining. Similarly, $c \ne b$ and so $c \in (a,b)$.  Then, for all sufficiently large $n$, the point $c_n$   belongs to $(a,b)$, which contradicts the assumption $(a,b) \cap\mathcal{D}=\emptyset$. Therefore $\mathcal{D}$ is dense in $[d_2,d_5]$. \qedhere
 	
%
\end{proof}
 \end{lemma}
 
Lemma \ref{denseset} can also be applied for the triple $(e_1,e_2,e_3)$. In $\mathbb{T}_{e_1}$, let $\mathcal{E}$ be the dense set generated by the points $e_2,e_3$.

\begin{lemma} \label{homomeomorphismdense} $\mathcal{D}$ is homeomorphic to  $\mathcal{E}$.
\begin{proof} Each point $x \in \mathcal{D}$ can be obtained from the points $d_i$ by finitely many geometric operations. Since we assume $(\sigma_n)$ converges on $d_i$,  $\lim_n \sigma_n(x)$ exists. Hence we can define $\overline{\sigma}: \mathcal{D} \rightarrow \mathcal{E}: x \mapsto \lim_n \sigma_n(x)$. In the following, we show that $\overline{\sigma}$ is a homeomorphism. By construction, $\overline{\sigma}$ is surjective. 
\begin{enumerate}[label=\arabic*),leftmargin=0pt,itemindent=*]
\item In the derived plane $\mathbb{T}_{d_1}$, let $\mathcal{R}$ be the closed rectangle formed by the parallel classes of $d_2$ and $d_3$; in $\mathbb{T}_{e_1}$, let $\mathcal{S}$  be the closed rectangle formed by the parallel classes of $e_2$ and $e_3$, cf. Figure \ref{fig:preservecontain}.   We first prove that if a point $p \in \mathcal{D}$ is in the interior of $\mathcal{R}$, then $q:= \overline{\sigma}(p)$ is in the interior of $\mathcal{S}$. 

\begin{figure}[h]
	\centering
	\begin{subfigure}[h]{0.5\textwidth}
		\centering
		\begin{tikzpicture}
		\draw [fill] (0,0) circle [radius=0.1];
		\node[below left] at (0,0) {$d_2$};
		\coordinate (d_2) at (0,0);
		
		\draw [fill] (0,7) circle [radius=0.1];
		\node[above left] at (0,7) {$d_4$};
		\coordinate (d_4) at (0,7);

		\draw [fill] (6,0) circle [radius=0.1];
		\node[below right] at (6,0) {$d_5$};
		\coordinate (d_5) at (6,0);

		\draw [fill] (6,7) circle [radius=0.1];
		\node[above right] at (6,7) {$d_3$};
		\coordinate (d_3) at (6,7);
		
		\draw [fill] (2,5) circle [radius=0.1];
		\node[above right] at (2,5) {$p_0$};
		\coordinate (p) at (2,5);
		
		\draw [fill] (6,5) circle [radius=0.1];
		\node[above right] at (6,5) {$p'$};
		\coordinate (p') at (6,5);
		
		\draw  (d_4) -- (d_3);
		\draw  (d_5) -- (d_3);
		\draw  (d_4) -- (d_2);
		\draw  (d_2) -- (d_5);
		
		\draw  (p) -- (p');
		
		
		\draw [name path=d_2--p'] (d_2) -- (p');
		\path [name path=d_4--p] (d_4) -- (6,1);
		\path [name intersections={of=d_4--p and d_2--p',by=r_1}];
		\draw  (d_4) -- (r_1);
		\node[left] at (r_1) {$r_1$};
		\draw [fill] (r_1) circle [radius=0.1];
		
		\path [name path=p+] (p) -- (2,0);
		\path [name intersections={of=p+ and d_2--p',by=r_0}];
		\node[left] at (r_0) {$r_0$};
		\draw [fill] (r_0) circle [radius=0.1];
		\draw [dashed] (p) -- (r_0);
		
		\path [name path=r_1+] (r_1) -- (3.805,5);
		\path [name path=p-] (p) -- (p');
		\path [name intersections={of=r_1+ and p-,by=p_1}];
		\node[above] at (p_1) {$p_1$};
		\draw [fill] (p_1) circle [radius=0.1];
		\draw  [dashed] (r_1) -- (p_1);
		\end{tikzpicture}
		\caption*{$\mathcal{R}$ in $\mathbb{T}_{d_1}$}
	\end{subfigure}
	\hfill
	\begin{subfigure}[h]{0.45\textwidth}
		\centering
		\begin{tikzpicture}
		\draw [fill] (0,0) circle [radius=0.1];
		\node[below left] at (0,0) {$e_2$};
		\coordinate (e_2) at (0,0);
		
		\draw [fill] (0,7) circle [radius=0.1];
		\node[above left] at (0,7) {$e_4$};
		\coordinate (e_4) at (0,7);

		\draw [fill] (6,0) circle [radius=0.1];
		\node[below right] at (6,0) {$e_5$};
		\coordinate (e_5) at (6,0);

		\draw [fill] (6,7) circle [radius=0.1];
		\node[above right] at (6,7) {$e_3$};
		\coordinate (e_3) at (6,7);
		
		\draw [fill] (0,5) circle [radius=0.1];
		\node[above right] at (0,5) {$q$};
		\coordinate (q) at (0,5);
		
		\draw [fill] (6,5) circle [radius=0.1];
		\node[above right] at (6,5) {$q'$};
		\coordinate (q') at (6,5);
		
		\draw  (e_4) -- (e_3);
		\draw  (e_5) -- (e_3);
		\draw  (e_4) -- (e_2);
		\draw  (e_2) -- (e_5);
		
		\draw  (q) -- (q');
		
		\draw (q') -- (e_2);
		\end{tikzpicture}
		\caption*{$\mathcal{S}$ in $\mathbb{T}_{e_1}$}
	\end{subfigure}
	\caption{} \label{fig:preservecontain}
\end{figure}

Each automorphism $\sigma_n$ induces an isomorphism between the derived $\mathbb{R}^2$-planes $\mathbb{T}_{d_1}$ and $\mathbb{T}_{\sigma_n(d_1)}$, cf. \cite[256]{gunter2001}. In particular,
$
\sigma_n([d_3,d_4]) = [\sigma_n(d_3), \sigma_n(d_4)],
$
cf. \cite[Theorem 3.5]{salzmann1967b} or \cite[Theorem 2.4.2]{gunter2001}.  Since $\pi(p, d_3) \in (d_3,d_4)$, we may assume $\overline{\sigma}(\pi(p, d_3)) \in [e_3,e_4]$.  It follows that $q \in \mathcal{S}$.

We show that $q$ cannot lie on the boundary of $\mathcal{S}$. Suppose for a contradiction that   $q \in [e_2,e_4)$. For inductive purposes, we let $p_0 \coloneqq p$, $p'\coloneqq \pi(d_3,p_0)$, and $r_0 \coloneqq [p_0]_+ \cap d_2p'$.  
For $n \ge 1$, let  $r_{n} \coloneqq d_4p_{n-1} \cap d_2p'$, and $p_{n}\coloneqq \pi(r_{n}, p)$. It follows that $r_{n+1} \in (r_n,p')$ and $p_{n+1} \in (p_n,p')$.  By applying Lemma \ref{lK4} for the sequence of circles $C_n\coloneqq \alpha(d_4, p_n,r_{n+1})$ and the constant sequence $(d_1)$ on $\mathbb{T}$,  we have both $(r_n)$ and $(p_n)$ convergent to $p'$.
 
	Let $q' \coloneqq \overline{\sigma}(p')$. 
	Since $\overline{\sigma}(p_0) = q$ and $\overline{\sigma}(r_0) = [q]_+ \cap e_2q'=e_2$. The inductive construction yields  $\overline{\sigma}(p_n) =q$ and $\overline{\sigma}(r_{n})= e_2$.
	
	Let $d_6 \coloneqq d_2d_3 \cap d_4d_5$ and  $d_6' \coloneqq \pi(d_6,p)$. Then there exists an $n \in \mathbb{N}$ such that $p_n \in (d_6',p')$. On the other hand, $\overline{\sigma}(p_n) \in e_2e_4$ implies $\overline{\sigma}(p_n) \not \in [\overline{\sigma}(d_6'),q']$, which is a contradiction. Hence, $q$ is in the interior of the rectangle $\mathcal{S}$.

\item By applying the argument in part 1) for arbitrary  rectangles in $\mathcal{D}$, we obtain: a) $\overline{\sigma}$ maps nonparallel points to nonparallel points; b) if $(x_n)  \in \mathcal{D}$ converges in $\mathbb{T}_{d_1}$, then $(\overline{\sigma}(x_n))$ converges in $\mathbb{T}_{e_1}$. In particular, $\overline{\sigma}$ is injective and continuous. The continuity of $\overline{\sigma}^{-1}$ now follows from interchanging the roles of $\mathcal{D}$ and $\mathcal{E}$ in part 1). This proves the lemma. \qedhere

\end{enumerate}
\end{proof}
\end{lemma}

Let $N$ be the subset of $\widetilde{\mathcal{P}^3}$ defined as
\begin{align*}
N \coloneqq \{ (\sigma(d_1),\sigma(d_2),\sigma(d_3)) \mid \sigma \in \Sigma \}.
\end{align*}

Let $\omega: \Sigma \rightarrow N$ be  defined as $\omega: \sigma \mapsto (\sigma(d_1),\sigma(d_2),\sigma(d_3))$.
From Lemma \ref{idmap} and the definition of $N$, it follows that $\omega$ is a continuous bijection. To prove that $\omega$ is a homeomorphism between $\Sigma$ and $N$, we rely on the following.

\begin{lemma}  \label{SigmacongN}If $(\sigma_n) \in \Sigma$ is a sequence such that $(\omega(\sigma_n))$ converges in $N$, then $(\sigma_n)$ converges in $\Sigma$. In particular, $\omega$ is a closed map and $N$ is closed in $\widetilde{\mathcal{P}^3}$.  
	\begin{proof}  
	\begin{enumerate}[label=\arabic*),leftmargin=0pt,itemindent=*]
		\item We first claim that $(\sigma_n)$ converges pointwise.  Fix $x \in \mathcal{P}$. Since $\mathcal{P}$ is compact, $(\sigma_n(x))$ has an accumulation point $x^*$. Passing to subsequences, we assume $(\sigma_n(x))$ converges to $x^*$.  By changing derived planes if necessary, we can further  assume that  $x^*$ is a  point in $\mathbb{T}_{e_1}$.

		We further assume that  $x \in (d_2,d_5)$ in  $\mathbb{T}_{d_1}$, as other cases can be treated similarly. Let $(y_r) \in \mathcal{D} \cap [d_2,x)$ and $(z_r) \in \mathcal{D} \cap (x,d_5]$ be two sequences convergent to $x$. Let $\overline{\sigma}: \mathcal{D} \rightarrow \mathcal{E}: x \mapsto \lim_n \sigma_n(x)$ as in the proof of Lemma \ref{homomeomorphismdense}. Passing to subsequences we assume $(\overline{\sigma}(y_r))\rightarrow y^*$ and $(\overline{\sigma}(z_r)) \rightarrow z^*$.
		
		In $\mathbb{T}_{e_1}$, by Lemma \ref{lcollinear}, for each fixed $r$, we have $x^* \in (\overline{\sigma}(y_r),\overline{\sigma}(z_r))$. Letting $r \rightarrow \infty,$ we get $x^* \in [y^*,z^*]$. Suppose for a contradiction that $y^* \ne z^*$. Since $\mathcal{E}$ is dense and $\overline{\sigma}$ is a bijection between $\mathcal{D}$ and $\mathcal{E}$, there exist two distinct  points $u,v \in \mathcal{D}$ such that $\overline{\sigma}(u) \ne \overline{\sigma}(v) \in (y^*,z^*)$. This implies  $\overline{\sigma}(u) \in (\overline{\sigma}(y_r),\overline{\sigma}(z_r))$ and thus  $u \in (y_r,z_r)$ for all  sufficiently large $r$.  The only point in the point set $\mathcal{P}$ satisfying this condition  is $x$, and so $u= x$. The same argument implies that $v = x$, which then yields a contradiction.
		Therefore $x^* = y^* = z^*$.  This proves the claim.  
		
		\item We now extend the map $\overline{\sigma}$ to the point set $\mathcal{P}$ by letting $\overline{\sigma} : \mathcal{P} \rightarrow \mathcal{P} : x \mapsto \lim_n \sigma_n(x)$, which is a well-defined map by part 1).  We claim that $(\sigma_n)$ converges uniformly to $\overline{\sigma}$.  It is sufficient to show  convergence on the  rectangle $\mathcal{R}$ formed by the parallel classes of $d_2$ and $d_3$ in the derived plane $\mathbb{T}_{d_1}$.

		Fix $\epsilon >0$.  For every $\xi \in  \mathcal{R}$ let $\mathcal{R}_{\xi}$ be a rectangle whose vertices $\xi_1,\xi_2,\xi_3,\xi_4$ belong to $\mathcal{D}$ such that $\xi$ is inside $\mathcal{R}_{\xi}$ and $\max_i \mathbf{d}(\overline{\sigma}(\xi),\overline{\sigma}(\xi_i)) \le \epsilon /2$, for $i=1,\ldots,4$.   The union of the interiors of  all such rectangles is an open cover of $\mathcal{R}$ and thus has a finite subcover $\mathscr{F}$ with the set of finite vertices  $\mathcal{D}_{\epsilon/2} \subset \mathcal{D}$. Since $(\sigma_n)$ converges pointwise, there exists $n_0$ such that if $ n \ge n_0$ and $ \chi \in \mathcal{D}_{\epsilon/2}$, then $\mathbf{d}(\sigma_n(\chi),\sigma(\chi)) \le \epsilon/2$.
		
		\begin{figure}[h]
			\begin{center}
				\begin{tikzpicture}
				\draw [fill] (0,0) circle [radius=0.1];
				\node[below left] at (0,0) {$\sigma_n(x_1)$};
				\coordinate (gnx_1) at (0,0);
				
				\draw [fill] (6,0) circle [radius=0.1];
				\node[below right] at (6,0) {$\sigma_n(x_2)$};
				\coordinate (gnx_2) at (6,0);
				
				\draw [fill] (6,7) circle [radius=0.1];
				\node[above right] at (6,7) {$\sigma_n(x_3)$};
				\coordinate (gnx_3) at (6,7);
				
				\draw [fill] (0,7) circle [radius=0.1];
				\node[above left] at (0,7) {$\sigma_n(x_4)$};
				\coordinate (gnx_4) at (0,7);
				
				\draw [fill] (1,2.9) circle [radius=0.1];
				\node[below] at (1,2.9) {$\sigma_n(x)$};
				\coordinate (gnx) at (1,2.9);
				
				\draw  (gnx_1) -- (gnx_2);
				\draw  (gnx_2) -- (gnx_3);
				\draw  (gnx_3) -- (gnx_4);
				\draw  (gnx_4) -- (gnx_1);
				\draw [fill] (3,2) circle [radius=0.1];
				\node[below] at (3.1,2) {$\sigma(x_1)$};
				\coordinate (gx_1) at (3,2);
				
				\draw [fill] (5,2) circle [radius=0.1];
				\node[below ] at (5,2) {$\sigma(x_2)$};
				\coordinate (gx_2) at (5,2);
				
				\draw [fill] (5,6) circle [radius=0.1];
				\node[above] at (5,6) {$\sigma(x_3)$};
				\coordinate (gx_3) at (5,6);
				
				\draw [fill] (3,6) circle [radius=0.1];
				\node[above] at (3,6) {$\sigma(x_4)$};
				\coordinate (gx_4) at (3,6);
				
				\draw [fill] (4.5,5) circle [radius=0.1];
				\node[above ] at (4.5,5) {$\sigma(x)$};
				\coordinate (gx) at (4.5,5);
				
				\draw  (gx_1) -- (gx_2);
				\draw  (gx_2) -- (gx_3);
				\draw  (gx_3) -- (gx_4);
				\draw  (gx_4) -- (gx_1);
				
				\draw  [dashed] (gx) -- (gx_1)  node [sloped,below right,pos=0.7, fill=white] {\footnotesize{$\le \varepsilon/2$}};
				\draw [dashed] (gnx_1) -- (gx_1) node [ sloped,below right,pos=0.4, fill=white] {\footnotesize{$\le \varepsilon/2$}};
				\draw  [dashed] (gnx) -- (gx)  node [ sloped,above left,pos=0.5, fill=white] {\footnotesize{$\le \varepsilon$}};
				\end{tikzpicture}
			\end{center}
			\caption{} \label{fig:uniformconvergence}
		\end{figure}
		
			Let $x \in \mathcal{R}$ and let $\mathcal{R}_x$ be a rectangle from $\mathscr{F}$  that covers $x$. Then 
			$$\max_i \mathbf{d}(\overline{\sigma}(x),\overline{\sigma}(x_i)) \le \epsilon /2,$$ 
			see Figure \ref{fig:uniformconvergence}. For $n \ge n_0$, we have $\mathbf{d}(\sigma_n(x_i),\overline{\sigma}(x_i)) \le \epsilon/2$, which implies $\max_i \mathbf{d}(\overline{\sigma}(x),\sigma_n(x_i)) \le \epsilon$. Since $\sigma_n(x)$ is in the interior of  the rectangle formed by  $\sigma_n(x_i)$, for $i=1,\dots,4$, we have $\mathbf{d}(\sigma_n(x),\overline{\sigma}(x)) \le \epsilon$. Therefore  $(\sigma_n)$ converges uniformly to $\overline{\sigma}$ on $\mathcal{R}$.    
		 
		 3) As in the  proof of Lemma \ref{homomeomorphismdense}, it follows that $\overline{\sigma}$ is a bijection and thus a homeomorphism. The continuity of geometric operations ensures that $\overline{\sigma}$ is an automorphism of $\mathbb{T}$.  Hence $(\sigma_n)$ converges to $\overline{\sigma}$ in $\Sigma$.  
		  \qedhere
		 
	\end{enumerate}
	\end{proof}
\end{lemma} 

	\begin{proof}[Proof of Theorem \ref{tcplie}] 
		We have shown in Lemma \ref{SigmacongN} that $\Sigma \cong N$. Since $N$ is  a closed subspace of $\widetilde{\mathcal{P}^3}$,  $\Sigma$ is locally compact with dimension at most 6. By Szenthe's Theorem (cf. \cite[Theorems A2.3.4 and A2.3.5]{gunter2001}), $\Sigma$ is a Lie group. Since $\Sigma$ is a subgroup of index at most 8 of the topological group $\text{\normalfont Aut}(\mathbb{T})$, it follows that $\text{\normalfont Aut}(\mathbb{T})$ is a Lie group of the same dimension as $\Sigma$.  \qedhere
		
	\end{proof}

%

	\begin{proof}[Proof of Theorem \ref{dimension4}] Let $\Gamma$ be the connected component of $\text{\normalfont Aut} (\mathbb{T})$ and let $\Delta^\pm$ be the kernel of the action of $\Gamma$ on $\mathcal{G}^\pm$.  
		
%
%
		
		1) Assume  $\dim T^+ = 3$. From the proof of \cite[Theorem 4.4.10]{gunter2001}, which does not use the Axiom of Touching,  $\mathbb{T}$ is isomorphic to a plane $\mathcal{M}(f,id)$, where $f$ is an orientation-preserving homeomorphism of $\mathbb{S}^1$. Hence $\mathbb{T}$ is a flat Minkowski plane. 
		
		2) If $\dim \text{\normalfont Aut}(\mathbb{T}) \ge 5$, then the same arguments as in the proof of \cite[Theorem 4.4.12]{gunter2001} show that $\mathbb{T}$ is isomorphic to the classical Minkowski plane.
		
		 Assume $\dim \text{\normalfont Aut}(\mathbb{T}) = 4$. 
		If at least one of  $\Gamma/\Delta^\pm$ is transitive on $\mathcal{G}^\pm$, then  $\mathbb{T}$ is isomorphic to a plane $\mathcal{M}(f,id)$, where $f$ is a semi-multiplicative homeomorphism of the form $f_{d,s}$, $(d,s) \ne (1,1)$ (cf. \cite[Theorem 4.4.15]{gunter2001}). Otherwise, $\Gamma$ fixes a point $p$ and acts transitively on $\mathcal{P} \backslash ([p]_+ \cup [p]_-)$. Then $\Gamma$ induces  a 4-dimensional point-transitive group of automorphisms $\overline{\Gamma}$ on the derived plane $\mathbb{T}_p$. By \cite[Theorem 4.12]{salzmann1967b}, $\mathbb{T}_p$ is isomorphic to either the Desarguesian plane or a half-plane. 
		
		We show $\mathbb{T}_p$ is not isomorphic to  a half-plane.
		Suppose the contrary. Since $\overline{\Gamma}$ acts transitively on the two sets of lines derived from $\mathcal{G}^\pm$, from the proof of \cite[Theorem 4.12]{salzmann1967b},    $\dim \Delta^\pm =3$. But then $\dim  \Delta^+ \Delta^-=6 >\dim \Gamma,$   a contradiction. 
		
		Hence $\mathbb{T}_p$ is a Desarguesian plane. It follows from \cite[Theorem 4.4.15]{gunter2001}  that $\mathbb{T}$ is a nonclassical generalised Hartmann plane. \qedhere
		
	\end{proof}

\printbibliography


\end{document}